\documentclass[reqno]{amsart}
\input{setup/preamble.sty}
\input{setup/macros.sty}

\begin{document}
    \title[Quasi-treeings are treeable]{Quasi-treeings are treeable: a streamlined proof}
    \author{Zhaoshen Zhai}
    \address{Department of Mathematics and Statistics, McGill University, 805 Sherbrooke Street West, Montreal, QC, H3A 0B9, Canada}
    \email{zhaoshen.zhai@mail.mcgill.ca}
    \thanks{This work was partially supported by McGill University's SURA (Science Undergraduate Research Award) grant.}
    \date{\today}

    \begin{abstract}
        We present a streamlined exposition of a construction by R. Chen, A. Poulin, R. Tao, and A. Tserunyan, which proves the treeability of equivalence relations generated by any locally-finite Borel graph such that each component is a quasi-tree. More generally, we show that if each component of a locally-finite Borel graph admits a \textit{finitely-separating Borel family of cuts}, then we may `canonically' replace each component of the graph by a tree of special ultrafilter-like objects on cuts called \textit{orientations}; moreover, if the cuts are \textit{dense towards ends}, then the union of these trees is a Borel treeing.
    \end{abstract}

    \maketitle

    The purpose of this note is to provide a streamlined proof of the main result in \cite{CPTT23} in order to better understand the general formalism developed therein. We attempt to make this note relatively self-contained, but nevertheless, we urge the reader to refer to the original paper for more detailed background/discussions and some generalizations of the results we have selected to include here.

    \subsection*{Treeings of equivalence relations}

    A \textit{countable Borel equivalence relation (CBER)} on a standard Borel space $X$ is a Borel equivalence relation $E\subseteq X^2$ with each class countable. We are interested in special types of \textit{graphings} of a CBER $E\subseteq X^2$, i.e. a Borel graph $G\subseteq X^2$ whose connectedness relation is precisely $E$. For instance, a graphing of $E$ such that each component is a tree is called a \textit{treeing} of $E$, and the CBERs that admit treeings are said to be \textit{treeable}. The main results of \cite{CPTT23} provide new sufficient criteria for treeability of certain classes of CBERs, and in particular, they prove the following

    \begin{mainTheorem}[Section \ref{sec:borel_treeings_of_graphings_with_dense_cuts}, \cite{CPTT23}*{Theorem 1.1}]\label{thm:treeing_quasi-trees}
        If a CBER $E$ admits a locally-finite graphing such that each component is a quasi-tree,\footnote{Recall that metric spaces $X$ and $Y$ are \textit{quasi-isometric} if they are isometric up to a bounded multiplicative and additive error, and $X$ is a \textit{quasi-tree} if it is quasi-isometric to a simplicial tree; see \cite{Gro93} and \cite{DK18}.} then $E$ is treeable.
    \end{mainTheorem}

    Roughly speaking, the existence of a quasi-isometry $G|C\to T_C$ to a simplicial tree $T_C$ for each component $C\subseteq X$ implies the existence of a collection $\mc{H}(C)\subseteq2^C$ of `cuts' (subsets $H\subseteq C$ with finite boundary such that both $H$ and $C\comp H$ are connected), which are `tree-like' in the sense that
    \begin{enumerate}
        \item[(i)] $\mc{H}(C)$ is \textit{finitely-separating}: each pair $x,y\in C$ is separated by finitely-many $H\in\mc{H}(C)$, and
        \item[(ii)] $\mc{H}(C)$ is \textit{dense towards ends}: $\mc{H}(C)$ contains a neighborhood basis for each end in $G|C$.
    \end{enumerate}
    Condition (i) allows for an abstract construction of a tree $\mc{U}^\circ(\mc{H}(C))$ whose vertices are special `ultrafilters'\footnote{As in \cite{CPTT23}, we call them \textit{orientations} instead, to avoid confusion with the more standard notion; see Definition \ref{def:orientation}.} on $\mc{H}(C)$, as outlined in the following diagram: starting from a finitely-separating family of cuts, one constructs a `dual median graph' $\mc{M}(\mc{H}(C))$ with said ultrafilters; this median graph has finite `hyperplanes', which allows one to apply a Borel cycle-cutting algorithm and obtain a `canonical' spanning tree thereof.

    \vspace{0.05in}
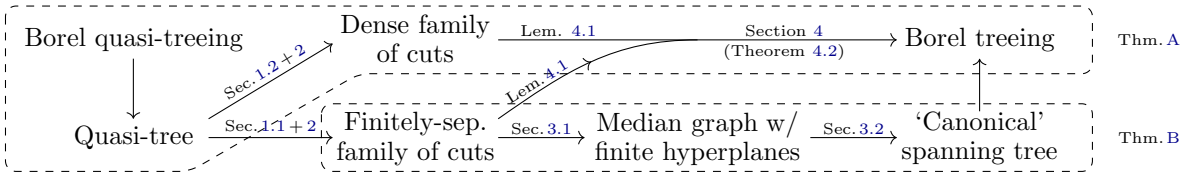
    \begin{center}
        \def\hsp{3.75}
        \def\vsp{1.3}
        \def\vsh{0.05}
        \begin{tikzpicture}
            \draw[dashed, rounded corners, opacity=0.2] (-1.7,0.5) -- (-1.7,-0.4-\vsp) -- (0.75,-0.4-\vsp) -- (3,-0.4) -- (3*\hsp+1.5,-0.4) -- (3*\hsp+1.5,0.5) -- cycle;
            \draw[dashed, rounded corners, opacity=0.2] (2.5,0.5-\vsp) rectangle (3*\hsp+1.5,-0.4-\vsp);

            \draw (3*\hsp+2.25,\vsh) circle (0in) node[opacity=0.7]{\tiny{Thm.\!\! \ref{thm:treeing_quasi-trees}}};
            \draw (3*\hsp+2.25,\vsh-\vsp) circle (0in) node[opacity=0.7]{\tiny{Thm.\!\! \ref{thm:component-wise_construction}}};

            \draw (0,\vsh) circle (0in) node{Borel quasi-treeing};
            \draw[->] (0,-0.2) -- (0,-\vsp+0.3);
            \draw (\hsp,0) circle (0in) node{$\begin{gathered} \textrm{Dense family}\\[-4pt] \textrm{of cuts} \end{gathered}$};
            \draw[->] (\hsp+1.1,\vsh) -- (3*\hsp-1.1,\vsh) node[xshift=-4.5cm, above=-2pt]{\tiny{Lem. \ref{lem:dual_median_graph_of_dense_cuts_locally_finite}}} node[xshift=-1.5cm, above=-2pt]{\tiny{Section \ref{sec:borel_treeings_of_graphings_with_dense_cuts}}} node[xshift=-1.5cm, below=-2pt]{\tiny{(Theorem \ref{thm:treeing_graphings_with_dense_cuts})}};
            \draw (3*\hsp,\vsh) circle (0in) node{Borel treeing};

            \draw (0,\vsh-\vsp) circle (0in) node{Quasi-tree};
            \draw[->] (1,\vsh-\vsp) -- (\hsp-1.1,\vsh-\vsp) node[midway, above=-2pt]{\tiny{Sec.\!\! \ref{sec:pocsets_of_cuts}\,+\,\ref{sec:graphs_with_dense_families_of_cuts}}};
            \draw[->] (1,-1) -- (\hsp-1.1,0) node[midway, above=-2pt, sloped]{\tiny{Sec.\!\! \ref{sec:ends_of_graphs}\,+\,\ref{sec:graphs_with_dense_families_of_cuts}}};
            \draw (\hsp,-\vsp) circle (0in) node{$\begin{gathered} \textrm{Finitely-sep.}\\[-4pt] \textrm{family of cuts} \end{gathered}$};
            \draw[->] (\hsp+1.1,\vsh-\vsp) -- (2*\hsp-1.5,\vsh-\vsp) node[midway, above=-2pt]{\tiny{Sec.\!\! \ref{sec:the_dual_median_graph_of_a_pocset}}};
            \draw[->-=0.5] (\hsp+1.1,-1) to [out=35, in=180] (2*\hsp,\vsh);
            \node[rotate=32] at (1.43*\hsp,-0.5) {\tiny{Lem.\!\! \ref{lem:dual_median_graph_of_dense_cuts_locally_finite}}};
            \draw (\hsp,-0.6) circle (0in);
            \draw (2*\hsp,-\vsp) circle (0in) node{$\begin{gathered} \textrm{Median graph w/}\\[-4pt] \textrm{finite hyperplanes} \end{gathered}$};
            \draw[->] (2*\hsp+1.5,\vsh-\vsp) -- (3*\hsp-1.1,\vsh-\vsp) node[midway, above=-2pt]{\tiny{Sec.\!\! \ref{sec:cycle_cutting_algorithm}}};
            \draw (3*\hsp,-\vsp) circle (0in) node{$\begin{gathered} \textrm{`Canonical'}\\[-4pt] \textrm{spanning tree} \end{gathered}$};
            \draw[->] (3*\hsp,0.4-\vsp) -- (3*\hsp,-0.2);
        \end{tikzpicture}
    \end{center}
    \vspace{-0.05in}
    Thus we have the following theorem, which can be viewed as a component-wise version of Theorem \ref{thm:treeing_quasi-trees}.

    \begin{mainTheorem}[Propositions \ref{prp:construction_of_dual_median_graph}, \ref{prp:dual_median_graph_of_cuts_has_finite_hyperplanes}, \ref{prp:canonical_spanning_trees}]\label{thm:component-wise_construction}
        For any finitely-separating family of cuts $\mc{H}$ on a connected locally-finite graph, its dual median graph $\mc{M}(\mc{H})$ has finite hyperplanes, and fixing a proper colouring of the intersection graph of those hyperplanes yields a canonical spanning tree $\mc{U}^\circ(\mc{H})$ of $\mc{M}(\mc{H})$.
    \end{mainTheorem}

    In the context of Theorem \ref{thm:treeing_quasi-trees}, for each component $C\subseteq G$ of a locally-finite quasi-treeing of $E$, the collection $\mc{H}(C)$ are cuts with bounded boundary diameter, and choosing a bound so that $\mc{H}(C)$ satisfies (i) and (ii) is a Borel operation since the \textit{a priori} coanalytic condition (ii) is, by Proposition \ref{prp:dense_cuts_iff_proper_walling}, equivalent to a Borel one. These conditions show that each $\mc{M}(\mc{H}(C))$ is locally-finite, so that $\bigsqcup_C\mc{U}^\circ(\mc{H}(C))$ can be encoded as a standard Borel space. This, combined with Theorem \ref{thm:component-wise_construction}, proves Theorem \ref{thm:treeing_quasi-trees}; see Section \ref{sec:borel_treeings_of_graphings_with_dense_cuts} for details.

    \subsection*{Acknowledgements}

    I would like to thank Professor Anush Tserunyan for supervising me for this project, for her consistent support, patience, and feedback, and for guiding me through this fulfilling research experience. I also thank the original authors of \cite{CPTT23} for their incredible work and detailed writing.

    \section{Preliminaries on Pocsets, Ends of Graphs, and Median Graphs}\label{sec:preliminaries}
    \renewcommand{\-}{\textrm{---}}

    \begin{notation*}
        A \textit{graph} on a set $X$ is a symmetric irreflexive binary relation $G\subseteq X^2$. For $A\subseteq X$, we say that $A$ is \textit{connected} if the induced subgraph $G[A]$ is. We always equip connected graphs with their \textit{path metric} $d$, and let $\Ball_r(x)$ be the closed ball of radius $r$ around $x$; more generally, we let $\Ball_r(A)\coloneqq\bigcup_{x\in A}\Ball_r(x)$.

        For a subset $A\subseteq X$, we let $\del_\mathsf{iv}A\coloneqq A\cap\Ball_1(\lnot A)$ be its \textit{inner vertex boundary}, $\del_\mathsf{ov}A\coloneqq\del_\mathsf{iv}(\lnot A)$ be its \textit{outer vertex boundary}, and let $\del_\mathsf{ie}A\coloneqq G\cap(\del_\mathsf{ov}A\times\del_\mathsf{iv}A)$ and $\del_\mathsf{oe}A\coloneqq\del_\mathsf{ie}(\lnot A)$ respectively be its \textit{inward} and \textit{outward edge boundaries}. Let $\del_\mathsf{v}A\coloneqq\del_\mathsf{iv}A\sqcup\del_\mathsf{ov}A$ be the \textit{(total) vertex boundary} of $A$.

        Finally, for $x,y\in X$, the \textit{interval} $[x,y]$ between $x,y$ is the union of all geodesics between $x,y$, consisting of exactly those $z\in X$ with $d(x,z)+d(z,y)=d(x,y)$. We say that $A\subseteq X$ is \textit{convex} if $[x,y]\subseteq A$ for all $x,y\in A$. For vertices $x,y,z\in X$, we write $x\-y\-z$ for $y\in[x,z]$. For all $w,x,y,z\in X$, observe that
        \begin{equation*}
            (w\-x\-y\textrm{ and }w\-y\-z)\ \ \ \ \Leftrightarrow\ \ \ \ (w\-x\-z\textrm{ and }x\-y\-z),
        \end{equation*}
        and both sides occur iff there is a geodesic from $w$ to $x$ to $y$ to $z$, which we write as $w\-x\-y\-z$.
    \end{notation*}

    \subsection{Profinite pocsets of cuts}\label{sec:pocsets_of_cuts}

    In the context of Theorem \ref{thm:treeing_quasi-trees}, the construction starts by identifying a profinite pocset $\mc{H}$ of `cuts' in each component of the graphing, which we first study abstractly. The finitely-separating subpocsets of $2^X$ are well-known in metric geometry as \textit{wallspaces}; see, e.g., \cite{Nic04} and \cite{CN05}.

    \begin{definition}\label{def:profinite_pocset}
        A \textit{pocset} $(\mc{H},\leq,\lnot,0)$ is a poset $(\mc{H},\leq)$ equipped with an order-reversing involution $\lnot:\mc{H}\to\mc{H}$ and a least element $0\neq\lnot0$ such that $0$ is the only lower-bound of $H,\lnot H$ for every $H\in\mc{H}$. We call the elements in $\mc{H}$ \textit{half-spaces}.

        A \textit{profinite pocset} is a pocset $\mc{H}$ equipped with a compact topology making $\lnot$ continuous and is \textit{totally order-disconnected}, in the sense that if $H\not\leq K$, then there is a clopen upward-closed $U\subseteq\mc{H}$ with $H\in U\not\ni K$.
    \end{definition}

    We are primarily interested in subpocsets of $(2^X,\subseteq,\lnot,\em)$ for a fixed set $X$, where $\lnot A\coloneqq A^c$ for $A\subseteq X$, which is profinite if equipped with the product topology of the discrete space $2$.

    \begin{remark*}
        We follow \cite{CPTT23}*{Convention 2.7}, where for a family $\mc{H}\subseteq2^X$ of subsets of a fixed set $X$, we write $\mc{H}^\ast\coloneqq\mc{H}\comp\l\{\em,X\r\}$ for the \textit{non-trivial} elements of $\mc{H}$.
    \end{remark*}

    The following proposition gives a sufficient criterion for subpocsets of $2^X$ to be profinite. We also show in this case that every non-trivial element $H\in\mc{H}^\ast$ is isolated, which will be important in Section \ref{sec:the_dual_median_graph_of_a_pocset}.

    \begin{lemma}\label{lem:finitely-separating_non-trivial_isolated}
        If $\mc{H}\subseteq2^X$ is a finitely-separating pocset, then $\mc{H}$ is closed and non-trivial elements are isolated.
    \end{lemma}
    \begin{proof}
        It suffices to show that the limit points of $\mc{H}$ are trivial, so let $A\in2^X\comp\l\{\em,X\r\}$. Fix $x\in A\not\ni y$. Since $\mc{H}$ is finitely-separating, there are finitely-many $H\in\mc{H}$ with $x\in H\not\ni y$, and for each such $H\in\mc{H}\comp\l\{A\r\}$, there is either some $x_H\in A\comp H$ or $y_H\in H\comp A$. The family of all subsets $B\subseteq X$ containing $x$ and each $x_H$, but not $y$ or any $y_H$, is then a clopen neighborhood of $A$ disjoint from $\mc{H}\comp\l\{A\r\}$, as desired.
    \end{proof}

    Our main method of identifying the finitely-separating pocsets in graphs is the following

    \begin{lemma}\label{lem:finitely_separating_iff_on_boundary_of_finite}
        Let $\mc{H}\subseteq2^X$ be a pocset in a connected graph $(X,G)$. If each $x\in X$ is on the vertex boundary of finitely-many half-spaces in $\mc{H}$, then $\mc{H}$ is finitely-separating. The converse holds too if $(X,G)$ is locally-finite.
    \end{lemma}
    \begin{proof}
        Any $H\in\mc{H}$ separating $x,y\in X$ separates some edge on any fixed path between $x$ and $y$, and there are only finitely-many such $H$ for each edge. If $(X,G)$ is locally-finite, then each $x\in X$ is separated from each of its finitely-many neighbors by finitely-many $H\in\mc{H}$.
    \end{proof}

    In the case that $\mc{H}$ is a pocset consisting of connected co-connected half-spaces with finite vertex boundary, finite-separation also controls the degree of `non-nestedness' of $\mc{H}$.

    \begin{definition}
        For a connected locally-finite graph $(X,G)$, we let $\mc{H}_\mathrm{conn}(X)$ and $\mc{H}_{\del<\infty}(X)$ respectively denote the subpocset of connected co-connected half-spaces in $2^X$ and the half-spaces in $2^X$ with finite-vertex boundary. A \textit{cut} in $(X,G)$ is a half-space $H\in\mc{H}_{\del<\infty}(X)\cap\mc{H}_\mathrm{conn}(X)$.
    \end{definition}

    \begin{definition}\label{def:nested}
        Let $\mc{H}\subseteq2^X$ be a pocset. Two half-spaces $H,K\in\mc{H}$ are \textit{nested} if $\lnot^iH\cap\lnot^jK=\em$ for some $i,j\in\l\{0,1\r\}$, where $\lnot^0H\coloneqq H$ and $\lnot^1H\coloneqq\lnot H$. We say that $\mc{H}$ is \textit{nested} if every pair $H,K\in\mc{H}$ is nested.
    \end{definition}

    \begin{lemma}\label{lem:connected_cuts_non_nested_finitely_others}
        For a pocset $\mc{H}$ of finitely-separating cuts, each $H\in\mc{H}$ is non-nested with finitely-many others.
    \end{lemma}
    \begin{proof}
        Fix $H\in\mc{H}$ and let $K\in\mc{H}$ be non-nested with $H$. By connectedness, the non-empty sets $H\cap K$ and $\lnot H\cap K$ are joined by a path in $K$, so $\del_\mathsf{v}H\cap K\neq\em$; similarly, $\del_\mathsf{v}H\cap\lnot K\neq\em$. For each $x\in\del_\mathsf{v}H\cap K$ and $y\in\del_\mathsf{v}H\cap\lnot K$, any fixed path $p_{xy}$ between them contains some $z\in\del_\mathsf{v}K\cap p_{xy}$; thus, any $K\in\mc{H}$ non-nested with $H$ contains some $z\in\del_\mathsf{v}K\cap p_{xy}$.

        Then, since there are finitely-many such $x,y\in\del_\mathsf{v}H$, for each of which there are finitely-many $z\in p_{xy}$, for each of which there are finitely-many $K\in\mc{H}$ with $z\in\del_\mathsf{v}K$ (by Lemma \ref{lem:finitely_separating_iff_on_boundary_of_finite}, since $\mc{H}$ is finitely-separating), there can only be finitely-many $K\in\mc{H}$ non-nested with $H$.
    \end{proof}

    \subsection{Ends of graphs}\label{sec:ends_of_graphs}

    Let $(X,G)$ be a connected locally-finite graph, and consider the Boolean algebra of finite vertex boundary half-spaces $\mc{H}_{\del<\infty}(X)\subseteq2^X$.

    \begin{definition}
        The \textit{end compactification} of $(X,G)$ is the Stone space $\widehat{X}$ of $\mc{H}_{\del<\infty}(X)$, whose non-principal ultrafilters are the \textit{ends} of $(X,G)$. We let $\epsilon(X)$ denote the set of ends of $(X,G)$.
    \end{definition}

    Identifying $X\into\widehat{X}$ via principal ultrafilter map $x\mapsto p_x$, we have $\epsilon(X)=\widehat{X}\comp X$. By definition, $\widehat{X}$ admits a basis of clopen sets containing $\widehat{A}\coloneqq\{p\in\widehat{X}\st A\in p\}$ for each $A\in\mc{H}_{\del<\infty}(X)$; abusing notation, we write $p\in A$, and say $A$ \textit{contains} $p$, when $p\in\widehat{A}$. Since $A\subseteq B$ iff $\widehat{A}\subseteq\widehat{B}$, we also write $p\in A\subseteq B$ for $p\in\widehat{A}\subseteq\widehat{B}$.

    \begin{lemma}\label{lem:infinite_iff_contains_end}
        A finite-boundary subset $A\in\mc{H}_{\del<\infty}(X)$ is infinite iff it contains an end in $(X,G)$.
    \end{lemma}
    \begin{proof}
        The converse direction follows since ends are non-principal. If $A$ is infinite, then by local-finiteness of $(X,G)$, K\H{o}nig's Lemma furnishes some infinite ray $(x_n)\subseteq A$. Then, $A$ is contained in the filter
        \begin{equation*}
            p\coloneqq\l\{H\in\mc{H}_{\del<\infty}(X)\st\fa^\infty n(x_n\in H)\r\},
        \end{equation*}
        which is ultra since $H\in p$ are of \textit{finite}-boundary, and is non-principal since it contains cofinite sets.
    \end{proof}

    \begin{definition}
        A pocset $\mc{H}\subseteq\mc{H}_{\del<\infty}(X)$ is \textit{dense towards ends} of $(X,G)$ if $\mc{H}$ contains a neighborhood basis for every end in $\epsilon(X)$.

        In other words, $\mc{H}$ is dense towards ends if for every $p\in\epsilon(X)$ and every (clopen) neighborhood $A\ni p$, where $A\in\mc{H}_{\del<\infty}(X)$, there is some $H\in\mc{H}$ with $p\in H\subseteq A$.
    \end{definition}

    We will show in Section \ref{sec:graphs_with_dense_families_of_cuts} that certain half-spaces $\mc{H}\subseteq\mc{H}_{\del<\infty}(X)$ induced by a locally-finite quasi-tree $(X,G)$ is dense towards ends. It will also be important that these half-spaces be cuts, in that witnesses to density can also be found in $\mc{H}\cap\mc{H}_\mathrm{conn}(X)$. The following lemma takes care of this.

    \begin{lemma}\label{lem:connected_witness_to_density}
        If a subpocset $\mc{H}\subseteq\mc{H}_{\del<\infty}$ is dense towards ends, then there is a subpocset $\mc{H}'\subseteq\mc{H}_{\del<\infty}\cap\mc{H}_\mathrm{conn}$, which is also dense towards ends, such that every $H'\in\mc{H}'$ has $\del_\mathsf{ie}H'\subseteq\del_\mathsf{ie}H$ for some $H\in\mc{H}$.
    \end{lemma}
    \begin{proof}
        A first attempt is to let $\mc{H}'$ be the connected components $H'_0$ of elements $H\in\mc{H}$, but this fails since $\lnot H'_0$ is not necessarily connected. Instead, we further take a component $\lnot H'$ of $\lnot H'_0$, which is co-connected since $H'$ contains $H'_0$ and the other components of $\lnot H'_0$, each of which is connected to $H'_0$ via $\del_\mathsf{ie}H'_0$. Formally,
        \begin{equation*}
            \mc{H}'\coloneqq\l\{H'\subseteq X\st H\in\mc{H}\textrm{ and }H'_0\in H/G\textrm{ and }\lnot H'\in\lnot H'_0/G\r\},
        \end{equation*}
        where $H/G$ denotes the $G$-components of $H$. Clearly $\del_\mathsf{ie}H'\subseteq\del_\mathsf{ie}H'_0\subseteq\del_\mathsf{ie}H$, and since $H'\in\mc{H}_\mathrm{conn}(X)$, it remains to show that $\mc{H}'$ is dense towards ends.

        Fix an end $p\in\epsilon(X)$ and a neighborhood $p\in A\in\mc{H}_{\del<\infty}(X)$. Let $B\supseteq\del_\mathsf{v}A$ be finite connected, which can be obtained by adjoining paths between its components. Then $\lnot B\in p$ since $p$ is non-principal, so there is $H\in\mc{H}$ with $p\in H\subseteq\lnot B$. Since $H\in\mc{H}_{\del<\infty}(X)$, it has finitely-many connected components, so exactly one of them belongs to $p$, say $p\in H'_0\subseteq H$. Note that $B\subseteq\lnot H\subseteq\lnot H'_0$, so since $B$ is connected, there is a unique component $\lnot H'\subseteq\lnot H'_0$ containing $B$.

        Observe that $H'\in\mc{H}'$ and $p\in H'$. Lastly, since $H'$ is connected and is disjoint from $\del_\mathsf{v}A\subseteq B$, and since $H'\subseteq\lnot A$ would imply $\lnot H'\in p$, this forces $H'\subseteq A$, and hence $p\in H'\subseteq A$ as desired.
    \end{proof}

    \subsection{Median graphs and projections}

    Starting from a profinite pocset $\mc{H}$ with every non-trivial element isolated, we construct in Section \ref{sec:the_dual_median_graph_of_a_pocset} its dual median graph $\mc{M}(\mc{H})$.

    We devote this section and the next to study some basic properties of median graphs and their projections, which will be used in Section \ref{sec:cycle_cutting_algorithm} to construct a spanning tree certain median graphs. For more comprehensive references of median graphs, and their general theory, see \cite{Rol98} and \cite{Bow22}.

    \begin{definition}
        A \textit{median graph} is a connected graph $(X,G)$ such that for any $x,y,z\in X$, the intersection
        \begin{equation*}
            [x,y]\cap[y,z]\cap[x,z]
        \end{equation*}
        is a singleton, whose element $\l\langle x,y,z\r\rangle$ is called the \textit{median} of $x,y,z$. Thus we have a ternary \textit{median} operation $\l\langle\cdot,\cdot,\cdot\r\rangle:X^3\to X$, and a \textit{median homomorphism} $f:(X,G)\to(Y,H)$ is a map preserving said operation.
    \end{definition}

    \begin{lemma}\label{lem:projections}
        For any $\em\neq A\subseteq X$ and $x\in X$, there is a unique point in $\cvx(A)$ between $x$ and every point in $A$, called the $\mathrm{projection}$ of $x$ towards $A$, denoted $\proj_A(x)$.

        Moreover, we have $\bigcap_{a\in A}[x,a]=[x,\proj_A(x)]$, and for any $y$ in this set, we have $\proj_A(y)=\proj_A(x)$.
    \end{lemma}
    \begin{proof}
        To show existence, pick any $a_0\in A$. Given $a_n\in\cvx(A)$, if there exists $a\in A$ with $a_n\not\in[x,a]$, set $a_{n+1}\coloneqq\l\langle x,a,a_n\r\rangle\in\cvx(A)$. Then $a_0\-a_1\-\cdots\-a_n\-x$ for all $n$, so this sequence terminates in at most $d(a_0,x)$ steps at a point in $\cvx(A)$ between $x$ and every point in $A$. For uniqueness, if there exist two such points $a,b\in\cvx(A)$, then $x\-a\-b$ and $x\-b\-a$, forcing $a=b$.

        Finally, if $x\-y\-\proj_A(x)$ and $a\in A$, then $x\-\proj_A(x)\-a$ and hence $x\-y\-a$. Conversely, let $x\-y\-a$ for all $a\in A$. Since $[y,a]\subseteq[x,a]$ for all $a$, we see that
        \begin{equation*}
            \proj_A(y)\in\cvx(A)\cap\bigcap_{a\in A}[y,a]\subseteq\cvx(A)\cap\bigcap_{a\in A}[x,a]
        \end{equation*}
        and hence $\proj_A(y)=\proj_A(x)$ by uniqueness. But since $y\-\proj_A(y)\-a$, we have $x\-y\-\proj_A(y)$, and hence $x\-y\-\proj_A(x)$ as desired.
    \end{proof}

    \begin{remark}\label{rem:projections}
        It follows from the proof above that for any median homomorphism $f:(X,G)\to(Y,H)$, we have $f(\proj_A(x))=\proj_{f(A)}(f(x))$ for any $\em\neq A\subseteq X$ and $x\in X$. Indeed, we have
        \begin{equation*}
            \proj_A(x)=\l\langle x,a_m,\dots,\l\langle x,a_2,\l\langle x,a_1,a_0\r\rangle\r\rangle\dots\r\rangle
        \end{equation*}
        for some $m\leq d(a_0,x)$ and $a_0,\dots,a_m\in A$, and this is preserved by $f$.
    
        For $A\coloneqq\l\{a,b\r\}$, we have $\proj_A(x)=\l\langle a,b,x\r\rangle$, and hence $\cvx(A)=\proj_A(X)=\l\langle a,b,X\r\rangle=[a,b]$.
    \end{remark}

    \begin{lemma}\label{lem:cones_are_convex}
        For each $x,y\in X$, $\cone_x(y)$ is convex, and if $xGy$, then $\cone_x(y)\sqcup\cone_y(x)=X$.
    \end{lemma}
    \begin{proof}
        Fix $a,b\in\cone_x(y)$ and $a\-c\-b$. It suffices to show that $x\-y\-\l\langle a,c,x\r\rangle$, for then $x\-y\-c$ since we have $x\-\l\langle a,c,x\r\rangle\-c$. Indeed, it follows from the following observations.
        \begin{itemize}
            \item $x\-y\-\l\langle a,b,x\r\rangle$, since $\l\langle a,b,x\r\rangle=\proj_{\l\{a,b\r\}}(x)$ and so $[x,\l\langle a,b,x\r\rangle]=[x,a]\cap[x,b]\ni y$ by Lemma \ref{lem:projections}.
            \item $x\-\l\langle a,b,x\r\rangle\-\l\langle a,c,x\r\rangle$, which follows from $\l\langle a,b,x\r\rangle\-\l\langle a,c,x\r\rangle\-a$, since $x\-\l\langle a,b,x\r\rangle\-a$ by definition. Indeed, we have $\l\langle a,c,x\r\rangle$ is in both $[a,x]$ and $[a,c]\subseteq[a,b]$, and since $\proj_{\l\{b,x\r\}}(a)=\l\langle a,b,x\r\rangle$, we have again by Lemma \ref{lem:projections} that $[\l\langle a,b,x\r\rangle,a]=[a,x]\cap[a,b]\ni\l\langle a,c,x\r\rangle$.
        \end{itemize}
        Finally, take $z\in X$ and consider $w\coloneqq\l\langle x,y,z\r\rangle\subseteq[x,y]$. Either $w=x$ or $w=y$ (but not both), giving us the desired partition.
    \end{proof}

    \begin{remark}\label{rem:projection_between_point_in_convex}
        In particular, this shows that if $xGy$, then $\cone_x(y)\in\mc{H}^\ast_\mathrm{cvx}(X)$. The convexity of cones also shows, in the situation of Lemma \ref{lem:projections}, that $\proj_A=\proj_{\cvx(A)}$, i.e., $\proj_A(x)$ is also between $x$ and every point in $\cvx(A)$: indeed, note that $\cone_x(\proj_A(x))$ is convex and contains $A$, so it contains $\cvx(A)$ too.
    \end{remark}

    \begin{lemma}\label{lem:projection_homomorphism}
        $\proj_A:X\onto\cvx(A)$ is a median homomorphism with $\proj_A\circ\cvx=\cvx\circ\proj_A$.
    \end{lemma}
    \begin{proof}
        The second claim follows from the first since, by Remark \ref{rem:projections}, we have
        \begin{equation*}
            f(\cvx(B))=f(\proj_B(X))=\proj_{f(B)}(f(X))=\cvx(f(B))
        \end{equation*}
        for all median homomorphisms $f:X\onto Y$ and $B\subseteq X$, so it in particular applies to $f\coloneqq\proj_A$.

        To this end, let $x\-y\-z\in X$ and set $w\coloneqq\l\langle\proj_A(x),\proj_A(y),\proj_A(z)\r\rangle\in\cvx(A)$. It suffices to show that $y\-w\-a$ for all $a\in A$, for then $w=\proj_A(y)$ and hence $\proj_A(x)\-\proj_A(y)\-\proj_A(z)$. But we have $y\-\proj_A(y)\-a$ already, so it further suffices to show that $y\-w\-\proj_A(y)$. For this, we note that
        \begin{equation*}
            x\-\proj_A(x)\-\proj_A(y)\ \ \ \ \mathrm{and}\ \ \ \ \proj_A(x)\-w\-\proj_A(y),
        \end{equation*}
        so $x\-w\-\proj_A(y)$, and similarly $z\-w\-\proj_A(y)$. Thus, it follows that
        \begin{equation*}
            \begin{aligned}
                w\in[\proj_A(y),x]\cap[\proj_A(y),z]&=[\proj_A(y),\proj_{\l\{x,z\r\}}(\proj_A(y))]\ \ \ \ && \textrm{Lemma \ref{lem:projections}} \\
                                                    &=[\proj_A(y),\proj_{\l[x,z\r]}(\proj_A(y))]\ \ \ \ &&\textrm{Remark \ref{rem:projection_between_point_in_convex}} \\
                                                    &\subseteq[\proj_A(y),y],\ \ \ \ && \textrm{Lemma \ref{lem:projections}}
            \end{aligned}
        \end{equation*}
        where the second equality follows from $\cvx(\l\{x,z\r\})=[x,z]$, and hence $\proj_{\l\{x,z\r\}}=\proj_{\l[x,z\r]}$.
    \end{proof}

    \subsection{Convex half-spaces of median graphs}\label{sec:convex_half-spaces_of_median_graphs}

    We now use projections to explore the geometry of convex half-spaces in a median graph $(X,G)$. For the axiomatics of convex structures, see \cite{vdV93}.

    For a convex co-convex half-space $H\in\mc{H}^\ast_\textrm{cvx}(X)$, we call the inward edge boundary $\del_\mathsf{ie}H$ a \textit{hyperplane}.

    \begin{proposition}\label{prp:half-spaces_are_cones}
        Each edge $(x,y)\in G$ is on a unique hyperplane, namely the inward boundary of $\cone_x(y)$, and conversely, each half-space $H\in\mc{H}^\ast_\mathrm{cvx}(X)$ is $\cone_x(y)$ for every $(x,y)\in\del_\mathsf{ie}H$.

        Thus, hyperplanes are equivalence classes of edges. Furthermore, this equivalence relation is generated by parallel sides of squares (i.e., $4$-cycles).
    \end{proposition}
    \begin{proof}
        We have $\cone_x(y)\in\mc{H}_\mathrm{cvx}^\ast(X)$ by Lemma \ref{lem:cones_are_convex}, and clearly $(x,y)\in\del_\mathsf{ie}\cone_x(y)$. Conversely, take $H\in\mc{H}_\mathrm{cvx}^\ast(X)$ and any $(x,y)\in\del_\mathsf{ie}H$. Then $H=\cone_x(y)$, for if $z\in H\cap\lnot\cone_x(y)$, then $z\in\cone_y(x)$, and hence $x\in[y,z]\subseteq H$ by convexity of $H$, a contradiction; if $z\in\cone_x(y)\cap\lnot H$, then $[x,z]\subseteq\lnot H$ by convexity of $\lnot H$, and hence $y\not\in H$, a contradiction.

        Finally, parallel edges of a strip of squares generate the same hyperplane since, for a given square, each vertex is between its neighbors and hence any hyperplane containing an edge contains its opposite edge. On the other hand, let $(a,b),(c,d)\in\del_\mathsf{ie}H$ for some $H\in\mc{H}_\mathrm{cvx}^\ast(X)$. Note that $\del_\mathsf{ov}H=\proj_{\lnot H}(H)$ is convex since $H$ is, and since $\proj_{\lnot H}$ preserves convexity by Lemma \ref{lem:projection_homomorphism}, any geodesic between $a,c\in\del_\mathsf{ov}H$ lies in $\del_\mathsf{ov}H$. Matching this geodesic via $\del_\mathsf{ie}H:\del_\mathsf{ov}H\to\del_\mathsf{iv}H$ gives us a geodesic between $b,d$ in $\del_\mathsf{iv}H$, which together with the matching forms the desired strip of squares.
    \end{proof}

    \begin{corollary}\label{cor:non-nested_iff_embedding_of_hamming}
        Two half-spaces $H,K\in\mc{H}_\mathrm{cvx}^\ast(X)$ are non-nested iff there is an embedding $\l\{0,1\r\}^2\into X$ of the Hamming cube into the four corners $\lnot^iH\cap\lnot^jK$.

        In particular, if $H,K\in\mc{H}_\mathrm{cvx}^\ast(X)$ are non-nested, then $\del_\mathsf{v}H\cap\del_\mathsf{v}K\neq\em$.
    \end{corollary}
    \begin{proof}
        Let $H,K$ be non-nested and take $x_1\in H\cap K$ and $x_2\in H\cap\lnot K$. Since $H$ is connected, any geodesic between $x_1,x_2$ crosses an edge $(x_1',x_2')\in\del_\mathsf{oe}K$ in $H$. Similarly, there is an edge $(y_1',y_2')\in\del_\mathsf{oe}K$ in $\lnot H$, so we may slide both edges along $\del_\mathsf{oe}K$ to obtain the desired square (see Proposition \ref{prp:half-spaces_are_cones}).

        Conversely, the half-spaces cutting the square are clearly non-nested.
    \end{proof}

    \begin{lemma}[Helly]\label{lem:helly}
        Any finite intersection of pairwise-intersecting non-empty convex sets is non-empty.
    \end{lemma}
    \begin{proof}
        For pairwise-intersecting convex sets $H_1,H_2,H_3$, pick any $x\in H_1\cap H_2$, $y\in H_1\cap H_3$ and $z\in H_2\cap H_3$; their median $\l\langle x,y,z\r\rangle$ then lies in $H_1\cap H_2\cap H_3$.

        Suppose that it holds for some $n\geq 3$ and let $H_1,\dots,H_{n+1}\subseteq X$ pairwise-intersect. Then $\l\{H_i\cap H_{n+1}\r\}_{i\leq n}$ is a family of $n$ pairwise-intersecting convex sets, so $\bigcap_{i\leq n+1}H_i=\bigcap_{i\leq n}(H_i\cap H_{n+1})$ is non-empty.
    \end{proof}

    Lastly, we have some useful finiteness conditions on convex half-spaces; the former implies that $\mc{H}_\mathrm{cvx}(X)$ is finitely-separating, and the latter allows us to replace finite sets with their convex hulls.

    \begin{lemma}\label{lem:half_space_separating_convex}
        Any two disjoint convex sets $\em\neq A,B\subseteq X$ can be separated by a half-space $A\subseteq H\subseteq\lnot B$, and furthermore we have $d(A,B)=|\!\l\{H\in\mc{H}_\mathrm{cvx}(X)\st A\subseteq H\subseteq\lnot B\r\}\!|$.
    \end{lemma}
    \begin{proof}
        Pick a geodesic $A\ni x_0Gx_1G\cdots Gx_n\in B$, where $n\coloneqq d(A,B)$. Then $H\coloneqq\cone_{x_1}(x_0)$, which is a half-space by Lemma \ref{lem:cones_are_convex}, separates $A,B$ since $x_0=\proj_A(x_n)$, and thus we have $A\subseteq\cone_{x_n}(x_0)\subseteq\cone_{x_1}(x_0)$ and $B\subseteq\cone_{x_0}(x_n)\subseteq\cone_{x_0}(x_1)$.

        Moreover, each such half-space $A\subseteq H\subseteq\lnot B$ satisfies $x_i\in H\not\ni x_{i+1}$ for a unique $i<n$, and conversely each pair $(x_i,x_{i+1})$ has a unique half-space separating them, so we have the desired bijection.
    \end{proof}

    \begin{lemma}\label{lem:convex_of_finite_is_finite}
        Every interval $[x,y]$ is finite. More generally, if $A\subseteq X$ is finite, then so is $\cvx(A)$.
    \end{lemma}
    \begin{proof}
        The singletons $\l\{x\r\}$ and $\l\{y\r\}$ are convex, so there are finitely-many half-spaces $H\subseteq[x,y]$. But each $z\in[x,y]$ is determined uniquely by those half-spaces containing it, so $[x,y]$ is finite.

        Let $A\coloneqq\l\{x_0,\dots,x_n\r\}$. Since $\cvx(A)=\proj_A(X)$, we have by Remark \ref{rem:projections} that points in $\cvx(A)$ are of the form $\l\langle x,x_n,\dots,\l\langle x,x_2,\l\langle x,x_1,x_0\r\rangle\r\rangle\dots\r\rangle$, which is finite by induction using that intervals are finite.
    \end{proof}

    \section{Graphs with Dense Families of Cuts}\label{sec:graphs_with_dense_families_of_cuts}

    Let $(X,G)$ be a connected locally-finite quasi-tree, which, in the context of Theorem \ref{thm:treeing_quasi-trees}, stands for a single component of the locally-finite graphing of the CBER. For Theorem \ref{thm:component-wise_construction} to apply, we need to first identify a family of finitely-separating cuts therein, and we do so in such a way that the cuts are dense towards ends.

    Since $X$ is a quasi-tree, and thus does not have arbitrarily long cycles, we expect that there is some finite bound $R<\infty$ such that the ends in $\epsilon(X)$ are `limits' of cuts $\mc{H}\coloneqq\mc{H}_{\diam(\del)\leq R}(X)\cap\mc{H}_\mathrm{conn}(X)$ with boundary diameter bounded by $R$. We show that this is indeed the case, in the sense that $\mc{H}$ is dense towards ends.

    \begin{lemma}\label{lem:coarse_equivalence_controls_boundary_diameter}
        If $f:(X,G)\to(Y,T)$ is a coarse-equivalence between connected graphs, then $\diam(\del_\mathsf{v}f^{-1}(H))$ is uniformly bounded in terms of $\diam(\del_\mathsf{v}H)$ for any $H\in\mc{H}_{\del<\infty}(Y)$.
    \end{lemma}
    \begin{proof}
        Since $f$ is coarse, there exists $S<\infty$ be such that $xGx'$ implies $d(f(x),f(x'))\leq S$, so that for any $(x,x')\in\del_\mathsf{ie}f^{-1}(H)$, there is a path of length $\leq S$ between $f(x)\not\in H$ and $f(x')\in H$. Thus both $d(f(x),\del_\mathsf{v}H)$ and $d(f(x'),\del_\mathsf{v}H)$ are bounded by $S$, so $f(\del_\mathsf{v}f^{-1}(H))\subseteq\Ball_S(\del_\mathsf{v}H)$ and hence
        \begin{equation*}
            \diam(f(\del_\mathsf{v}f^{-1}(H)))\leq\diam(\del_\mathsf{v}H)+2S.
        \end{equation*}
        That $f$ is a coarse-\textit{equivalence} gives us a uniform bound of $\diam(\del_\mathsf{v}f^{-1}(H))$ in terms of $\diam(\del_\mathsf{v}H)$.
    \end{proof}

    In particular, if $\diam(\del_\mathsf{v}H)$ is itself also uniformly bounded, then so is $\diam(\del_\mathsf{v}f^{-1}(H))$.

    \begin{proposition}\label{prp:invariant_of_density_coarse_equivalence}
        The class of connected locally-finite graphs in which $\mc{H}_{\diam(\del)\leq R}$ is dense towards ends for some $R<\infty$ is invariant under coarse-equivalence.
    \end{proposition}
    \begin{proof}
        Let $(X,G)$, $(Y,T)$ be connected locally-finite graphs, $f:X\to Y$ be a coarse-equivalence with quasi-inverse $g:Y\to X$, and suppose $\mc{H}_{\diam(\del)\leq S}(Y)$ is dense towards ends for some $S<\infty$. By Lemma \ref{lem:coarse_equivalence_controls_boundary_diameter}, pick some $R<\infty$ so that for any $H\in\mc{H}_{\diam(\del)\leq S}(Y)$, we have $f^{-1}(H)\in\mc{H}_{\diam(\del)\leq R}(X)$.

        Fix an end $p\in\epsilon(X)$ and a neighborhood $p\in A\in\mc{H}_{\del<\infty}(X)$. We need to find some $B\in\mc{H}_{\del<\infty}(Y)$ such that $f(p)\in B$ and $f^{-1}(B)\subseteq A$, for then $f(p)\in H$ for some $B\supseteq H\in\mc{H}_{\diam(\del)\leq S}(Y)$, and hence we have
        \begin{equation*}
            p\in f^{-1}(H)\subseteq f^{-1}(B)\subseteq A
        \end{equation*}
        with $f^{-1}(H)\in\mc{H}_{\diam(\del)\leq R}(X)$. For convenience, let $D<\infty$ be the uniform distance $d(1_X,g\circ f)$.

        To this end, note that $f(p)\in B$ iff $p\in f^{-1}(B)$. Since $p\in A$, the latter can occur if $|A\symdiff f^{-1}(B)|<\infty$, and so we need to find such a $B\in\mc{H}_{\del<\infty}(Y)$ with the additional property that $f^{-1}(B)\subseteq A$.
        \begin{leftbar}
        \textit{Attempt 1.} Set $B\coloneqq g^{-1}(A)\in\mc{H}_{\del<\infty}(Y)$. Then $f^{-1}(B)\subseteq\Ball_D(A)$ since if $(g\circ f)(x)\in A$, then
                \begin{equation*}
                    d(x,A)\leq d(x,(g\circ f)(x))\leq d(1_X,g\circ f)=D.
                \end{equation*}
            By local-finiteness of $G$, we see that $A\symdiff f^{-1}(B)=A\comp f^{-1}(B)$ is finite, as desired.
        \end{leftbar}
        However, it is \textit{not} the case that $f^{-1}(B)\subseteq A$. To remedy this, we `shrink' $A$ by $D$ to $A'$ so that $\Ball_D(A')\subseteq A$, and take $B\coloneqq g^{-1}(A')$ instead. Indeed, $A'\coloneqq\lnot\Ball_D(\lnot A)\subseteq A$ works, since $f^{-1}(B)\subseteq\Ball_D(A')$ as before, so $A'\symdiff f^{-1}(B)=A'\comp f^{-1}(B)$ is finite. Also, $A\symdiff A'$ is finite since $x\in A\symdiff A'$ iff $x\in A$ and $d(x,\lnot A)\leq D$, so $A\symdiff f^{-1}(B)$ is finite too. It remains to show that $\Ball_D(A')\subseteq A$, for then $f^{-1}(B)\subseteq A$ as desired.

        Indeed, if $y\in\Ball_D(A')$, then by the (reverse) triangle-inequality we have $d(y,\lnot A)\geq d(x,\lnot A)-d(x,y)$ for all $x\in A'$. But $d(x,\lnot A)>D$, strictly, so $d(y,\lnot A)>D-D=0$, and hence $y\in A$.
    \end{proof}

    \begin{corollary}\label{cor:density_of_bounded_diameter_boundary_cuts_quasi_tree}
        If $(X,G)$ is a locally-finite quasi-tree, then the subpocset $\mc{H}_{\diam(\del)\leq R}(X)\cap\mc{H}_\mathrm{conn}(X)$ is dense towards ends for some $R<\infty$.
    \end{corollary}
    \begin{proof}
        Observe that $\mc{H}_{\diam(\del)\leq2}(T)$ is dense towards ends for any tree $T$, so Proposition \ref{prp:invariant_of_density_coarse_equivalence} proves the density of $\mc{H}_{\diam(\del)\leq R}(X)$ for some $R<\infty$. By Lemma \ref{lem:connected_witness_to_density}, there is a subpocset $\mc{H}'\subseteq\mc{H}_{\del<\infty}(X)\cap\mc{H}_\mathrm{conn}(X)$ dense towards ends such that for every $H'\in\mc{H}'$, we have $\del_\mathsf{ie}H'\subseteq\del_\mathsf{ie}H$ for some $H\in\mc{H}_{\diam(\del)\leq R}(X)$. Hence we have $H'\in\mc{H}_{\diam(\del)\leq R}(X)\cap\mc{H}_\mathrm{conn}(X)$, so the result follows.
    \end{proof}

    The cuts $\mc{H}\coloneqq\mc{H}_{\diam(\del)\leq R}(X)\cap\mc{H}_\mathrm{conn}(X)$ obtained here will be our starting point for Theorem \ref{thm:component-wise_construction}, so we need to show that it is finitely-separating. Indeed, by local-finiteness of $(X,G)$, the $R$-ball around any fixed $x\in X$ is finite, so since any cut $H\in\mc{H}$ with $x\in\del_\mathsf{v}H$ is contained in said $R$-ball, there are finitely-many such cuts. Thus, by Lemma \ref{lem:finitely_separating_iff_on_boundary_of_finite}, $\mc{H}$ is finitely-separating.

    Furthermore, we have by Lemma \ref{lem:finitely-separating_non-trivial_isolated} that $\mc{H}\subseteq2^X$ is closed and every non-trivial element is isolated, and those conditions allow for the construction in Theorem \ref{thm:component-wise_construction} to continue in Section \ref{sec:the_dual_median_graph_of_a_pocset}.

    Finally, we will need to represent `density towards ends' in a Borel manner, which ultimately is to ensure that the bounds $R<\infty$ can be obtained uniformly across all components; see Section \ref{sec:borel_treeings_of_graphings_with_dense_cuts} for details. For this, we will need the following

    \begin{definition}\label{def:H_blocks}
        Let $\mc{H}\subseteq2^X$ be a pocset. For $x,y\in X$, write $x\sim_\mc{H}y$ if $x$ and $y$ are contained in exactly the same half-spaces in $\mc{H}$, which induces an equivalence relation on $X$ whose classes $[x]_\mc{H}$ are called \textit{$\mc{H}$-blocks}.
    \end{definition}

    \begin{proposition}\label{prp:dense_cuts_iff_proper_walling}
        If $\mc{H}$ is a finitely-separating pocset of cuts on a connected locally-finite graph, then $\mc{H}$ is dense towards ends iff (i) each $\mc{H}$-block is finite, and (ii) each $H\in\mc{H}^\ast$ has only finitely-many successors.
    \end{proposition}
    \begin{proof}
        If $\mc{H}$ is dense towards ends, we will cover (certain closed sets in) $\epsilon(X)$ by cuts in $\mc{H}$, which contains a finite subcover. We will show that a certain Boolean combination of this subcover, which has finite-boundary, is finite using Lemma \ref{lem:infinite_iff_contains_end}, giving us the desired finiteness claims (i) and (ii). Below are the details.

        For (i), fix an $\mc{H}$-block $[x]_\mc{H}$ and let $A\coloneqq\lnot\l\{x\r\}$. By density, we can cover each end $p\in\epsilon(X)$ with some $H_p\subseteq A$, which gives us a finite subcover $\l\{H_i\r\}_{i<n}$ of $\epsilon(X)$. Note that $\bigcap_{i<n}\lnot H_i\in\mc{H}_{\del<\infty}(X)$ contains $[x]_\mc{H}$, and is finite since it contains no ends in $\epsilon(X)$.

        For (ii), fix $H\in\mc{H}^\ast$ and let $\mc{K}\subseteq\mc{H}^\ast$ be the collection of all successors of $H$. By density, we can cover each end $p\in\epsilon(\lnot H)$ by some $H_p\subseteq\lnot H$ in $\mathcal{H}$, which in turn is contained in $\lnot K_\alpha$ for some $K_\alpha\in\mc{K}$; this gives us a finite subcover $\l\{\lnot K_i\r\}_{i<n}$ of $\epsilon(\lnot H)$. Every successor $K$ of $H$ not in $\l\{K_i\r\}_{i<n}$ is non-nested with at least one $K_i$, and by Lemma \ref{lem:connected_cuts_non_nested_finitely_others}, each $K_i$ is non-nested with finitely-many other half-spaces. Thus $\mc{K}$ is finite.\\[-0.1in]\ 

        Conversely, fix an end $p\in\epsilon(X)$ and a neighborhood $p\in A\in\mc{H}_{\del<\infty}(X)$. Since $\mc{H}$ consists of connected sets, it suffices to find some $H\in\mc{H}$ containing $\del_\mathsf{v}A$ but not $p$, for then $p\in\lnot H\subseteq A$ as desired.
        \begin{center}
            \begin{minipage}{0.95\textwidth}
                \begin{observation*}
                    Finitely-many $H,\lnot H\in\mc{H}$ may be removed from $\mc{H}$ and it will still satisfy the conditions (i) and (ii) above. Indeed, that (ii) still holds is obvious. For (i), we may remove a single pair, since with $\mc{H}'\coloneqq\mc{H}\comp\l\{H,\lnot H\r\}$, the map $X/\mc{H}\to X/\mc{H}'$ sending $[x]_\mc{H}\to[x]_{\mc{H}'}$ is surjective and at-most 2-to-1.
                \end{observation*}
            \end{minipage}
        \end{center}
        Thus, for each of the finitely-many $x,y\in\del_\mathsf{v}A$, we may remove the finitely-many half-spaces separating them, so we may assume that $\mc{H}=\mc{H}_A\sqcup\lnot\mc{H}_A$ is partitioned into the half-spaces in $\mc{H}_A$ containing $\del_\mathsf{v}A$, and its complements which are disjoint from $\del_\mathsf{v}A$. Towards a contradiction, assume that each $H\in\mc{H}_A$ contains $p$.

        Since $\mc{H}$-blocks are finite, there exists some $H_0\in\mc{H}^\ast_A$, which we may take to be minimal by Lemma \ref{lem:finitely-separating_non-trivial_isolated}. There are finitely-many half-spaces non-nested with $H_0$ by Lemma \ref{lem:connected_cuts_non_nested_finitely_others}, so by the above observation, we may assume without loss of generality that there are none. By minimality of $H_0$, all successors $K\supset\lnot H_0$ lie in $\mc{H}^\ast_A$, so the finite intersection $B\coloneqq H_0\cap\bigcap_{K\supset\lnot H_0}K$ contains $p$ and is hence infinite. Since $\mc{H}$-blocks are finite and $B$ is a union of $\mc{H}$-blocks, it contains infinitely-many $\mc{H}$-blocks; any two such $\mc{H}$-blocks are separated by some $H\in\mc{H}_A^\ast$ nested with $H_0$, which forces $H\subset H_0$, and contradicts the minimality of $H_0$.
    \end{proof}

    \section{The Dual Median Graph of a Profinite Pocset and its Spanning Trees}

    \subsection{Construction of the dual median graph}\label{sec:the_dual_median_graph_of_a_pocset}

    We present a classical construction in geometric group theory of a median graph $\mc{M}(\mc{H})$ associated to a profinite pocset $\mc{H}$ with every non-trivial element isolated; see \cite{Dun79}, \cite{Rol98}, \cite{Sag95}, and \cite{NR03} for various other applications of this construction.

    In the context of Theorem \ref{thm:treeing_quasi-trees}, this will be applied to the pocset $\mc{H}\coloneqq\mc{H}_{\diam(\del)\leq R}(X)\cap\mc{H}_\mathrm{conn}(X)$ of cuts in a locally-finite graph $(X,G)$, and will also be the first step in the construction in Theorem \ref{thm:component-wise_construction}.
    
    \begin{definition}\label{def:orientation}
        An \textit{orientation} on $\mc{H}$ is an upward-closed subset $U\subseteq\mc{H}$ containing exactly one of $H,\lnot H$ for each $H\in\mc{H}$. We let $\mc{U}(\mc{H})$ denote the set of all orientations on $\mc{H}$ and let $\mc{U}^\circ(\mc{H})$ denote the clopen ones.

        Intuitively, an orientation is a `maximally consistent' choice of half-spaces\footnote{This can be formalized by letting $\sim$ be the equivalence relation on $\mc{H}$ given by $H\sim\lnot H$. Letting $\del:\mc{H}\to\mc{H}/\!\!\sim$ denote the quotient map, orientations $U\subseteq\mc{H}$ then correspond precisely to sections $\phi:\mc{H}/\!\!\sim\,\to\mc{H}$ of $\del$ such that $\phi(\del H)\not\subseteq\lnot\phi(\del K)$ for every $H,K\in\mc{H}$; the latter condition rules out `orientations' of the form $\leftarrow\hspace{-4.15pt}|\,\,|\hspace{-4.15pt}\rightarrow$.} in $\mc{H}$.
    \end{definition}

    \begin{example}
        Each $x\in X$ induces its \textit{principal orientation} $\widehat{x}\coloneqq\l\{H\in\mc{H}\st x\in H\r\}$, which is clopen in $\mc{H}$, and gives us a \textit{principal orientations} map $X\to\mc{U}^\circ(\mc{H})$. However, this map is \textit{not necessarily} injective, and its fibers $[x]_\mc{H}\coloneqq\l\{y\in X\st\widehat{x}=\widehat{y}\r\}$ are precisely the $\mc{H}$-blocks as in Definition \ref{def:H_blocks}.
    \end{example}

    \begin{proposition}\label{prp:construction_of_dual_median_graph}
        Let $\mc{H}$ be a profinite pocset with every non-trivial element isolated. Then the graph $\mc{M}(\mc{H})$, whose vertices are clopen orientations $\mc{U}^\circ(\mc{H})$ and whose edges are pairs $\l\{U,V\r\}$ with $V=U\symdiff\l\{H,\lnot H\r\}$ for some minimal $H\in U\comp\l\{\lnot0\r\}$, is a median graph with path metric $d(U,V)=|U\symdiff V|/2$ and medians
        \begin{equation*}
            \begin{aligned}
                \l\langle U,V,W\r\rangle&\coloneqq\l\{H\in\mc{H}\st H\textrm{ belongs to at least two of }U,V,W\r\} \\
                                        &=(U\cap V)\cup(V\cap W)\cup(U\cap W).
            \end{aligned}
        \end{equation*}
    \end{proposition}

    \begin{proof}
        First, $V\coloneqq U\symdiff\l\{H,\lnot H\r\}$ as above is clopen since $H,\lnot H\in\mc{H}^\ast$ are isolated (whence $\l\{H\r\},\l\{\lnot H\r\}$ are clopen), and it is an orientation by minimality of $H$. That $\mc{M}(\mc{H})$ is connected follows from the following claim by noting that $U\symdiff V\not\ni0,\lnot0$ is clopen, so it is a compact set of isolated points, whence finite.
        \begin{center}
            \begin{minipage}{0.95\textwidth}
                \begin{claim*}[\cite{Sag95}*{Theorem 3.3}]
                    There is a path between $U,V$ iff $U\symdiff V$ is finite, in which case
                    \begin{equation*}
                        d(U,V)=|U\symdiff V|/2=|U\comp V|=|V\comp U|.
                    \end{equation*}
                \end{claim*}
                \begin{proof}
                    If $(U_i)_{i<n}$ is a path from $U\eqqcolon U_0$ to $V\eqqcolon U_{n-1}$, then, letting $\l\{H_i,\lnot H_i\r\}\coloneqq U_i\symdiff U_{i-1}$ for all $1\leq i<n$ gives us a sequence $(H_i)_{i<n}$ inducing\footnote{In the sense that $U_i=U_{i-1}\symdiff\l\{H_i,\lnot H_i\r\}$ and $H_i\in U_i$ for each $1\leq i<n$; see \cite{Tse20}*{Definition 2.20}.} this path, whence $U\symdiff V$ consists of $\l\{H_i\r\}_{i<n}$ and their complements. Thus $U\symdiff V=2n=2d(U,V)$, as desired.

                    \hspace{0.2in}Conversely, if $U\symdiff V=\l\{H_1,\dots,H_n\r\}\sqcup\l\{K_1,\dots,K_m\r\}$ with $U\comp V=\l\{H_i\r\}$ and $V\comp U=\l\{K_j\r\}$, then $\lnot H_i\in V\comp U$ and $\lnot K_j\in U\comp V$ for all $i<n$ and $j<m$, so $n=m$ and $V=U\cup\l\{\lnot H_i\r\}\comp\l\{H_i\r\}$. We claim that there is a permutation $\sigma\in S_n$ such that $(H_{\sigma(i)})$ induces a path from $U\eqqcolon U_0$ to $V\eqqcolon U_{n-1}$. Choose a minimal $H\in\l\{H_i\r\}_{i\leq n}$, which is also minimal in $U$: if $K\subseteq H$ for some $K\in U$, then $\lnot H\subseteq\lnot K$, and hence $\lnot K\in V$, so $K=H_i\subseteq H$ for some $i$, forcing $K=H$. Set $U_1\coloneqq U\symdiff\l\{H,\lnot H\r\}$, which is a clopen orientation, and continuing inductively gives us the desired path with $d(U,V)=n$.\phantom\qedhere\hfill$\square$
                \end{proof}
            \end{minipage}
        \end{center}
        Finally, we show that $\mc{M}(\mc{H})$ is a median graph. Fix $U,V,W\in\mc{U}^\circ(\mc{H})$, and note that for any $M\in\mc{U}^\circ(\mc{H})$, we have by the triangle inequality that $M\in[U,V]$ iff $(U\comp M)\cup(M\comp V)\subseteq U\comp V$, which clearly occurs iff $U\cap V\subseteq M\subseteq U\cup V$. Thus, a vertex $M$ lies in the triple intersection $[U,V]\cap[V,W]\cap[U,W]$ iff
        \begin{equation*}
            (U\cap V)\cup(V\cap W)\cup(U\cap W)\subseteq M\subseteq(U\cup V)\cap(V\cup W)\cap(U\cup W).
        \end{equation*}
        Note that the two sides coincide, so $M=\l\langle U,V,W\r\rangle$ $-$ which is clopen if $U,V,W$ are $-$ is as claimed.
    \end{proof}

    Given such a pocset $\mc{H}$, the graph $\mc{M}(\mc{H})$ constructed above is called the \textit{dual}\footnote{The name comes from a Stone-type duality between the categories \{median graphs, median homomorphisms\} and \{profinite pocsets with non-trivial points isolated, continuous maps\}, where from a median graph $X$ one can construct a canonical pocset $\mc{H}_\mathrm{cvx}(X)$ of convex half-spaces (see \cite{CPTT23}*{Section 2.D} for details).} median graph of $\mc{H}$. An important special case of this construction is when $\mc{H}$ is \textit{nested}, in which case $\mc{M}(\mc{H})$ is a tree.

    \begin{corollary}\label{cor:nested_implies_tree}
        Let $\mc{H}$ be a profinite pocset with non-trivial points isolated. If $\mc{H}$ is nested, then the median graph $\mc{M}(\mc{H})$ is acyclic, and hence $\mc{M}(\mc{H})$ is a tree.
    \end{corollary}
    \begin{proof}
        Let $(U_i)_{i<n}$ be a cycle in $\mc{M}(\mc{H})$, say induced by some sequence $(H_i)_{i<n}\subseteq\mc{H}$ of half-spaces. We show that $(H_i)$ is \textit{strictly} increasing, so that $H_0\subset\cdots\subset H_n\subset H_0$, which is absurd.

        We have $U_{i+1}=U_{i-1}\cup\{\lnot H_i,\lnot H_{i-1}\}\comp\l\{H_i,H_{i-1}\r\}$, so $H_i\neq\lnot H_{i-1}$ (for otherwise $U_{i+1}=U_{i-1}$). Since $H_i\in U_i=U_{i-1}\cup\{\lnot H_{i-1}\}\comp\l\{H_{i-1}\r\}$, we see that $H_i\in U_{i-1}$, and since $H_i\neq H_{i-1}$, it suffices by nestedness of $\mc{H}$ to remove the three cases when $\lnot H_i\subseteq H_{i-1}$, $H_{i-1}\subseteq\lnot H_i$, and $H_i\subseteq H_{i-1}$.

        Indeed, if $\lnot H_i\subseteq H_{i-1}$, then $H_{i-1}\in U_{i+1}$ by upward-closure of $U_{i+1}\ni\lnot H_i$. But since $H_{i-1}\neq\lnot H_i$, we have by definition of $U_{i+1}$ that $H_{i-1}\in U_i$, a contradiction. The other cases are similar.
    \end{proof}

    Nonetheless, in the general non-nested case, $\mc{M}(\mc{H})$ still admits a \textit{canonical} spanning tree if we fix a proper colouring of $\mc{H}_\mathrm{cvx}^\ast(\mc{M}(\mc{H}))$ into its nested sub-pocsets (Proposition \ref{prp:canonical_spanning_trees}). For this, we will need the following

    \begin{proposition}\label{prp:dual_median_graph_of_cuts_has_finite_hyperplanes}
        The dual median graph $\mc{M}(\mc{H})$ of a pocset of finitely-separating cuts has finite hyperplanes.
    \end{proposition}
    \begin{proof}
        Fix $K\in\mc{H}^\ast_\mathrm{cvx}(\mc{M}(\mc{H}))$, which by Proposition \ref{prp:half-spaces_are_cones} is of the form $K=\cone_V(U)$ for some (and hence any) $(U,V)\in\del_\mathsf{ie}K$, and we have by Proposition \ref{prp:construction_of_dual_median_graph} that $V=U\symdiff\l\{H,\lnot H\r\}$ for some (non-trivial) minimal $H\in U$. We claim that any other edge $(U',V')\in\del_\mathsf{ie}K$ can be reached from $(U,V)$ by simultaneously flipping only the half-spaces $H'_0,\dots,H'_n\in\mc{H}^\ast$ non-nested with $H$, of which there are finitely-many by Lemma \ref{lem:connected_cuts_non_nested_finitely_others}.

        Since the edges $(U,V),(U',V')$ induce the same hyperplane $\del_\mathsf{ie}K$, it suffices by Proposition \ref{prp:half-spaces_are_cones} to prove this for when $(U',V')$ is an edge of a square parallel to $(U,V)$, in which case there is some minimal $H'\in\mc{H}^\ast$ flipping both  $U$ to $U'$ and $V$ to $V'$. Note that $H,H'$ are non-nested by minimality, and $H\in U'$ is still minimal since $U'=U\symdiff\l\{H',\lnot H'\r\}$, and the only way it is not is if $\lnot H'\subseteq H$, which contradicts the minimality of $H$. Thus, $H$ induces the edge from $U'$ to $V'$ with $H'$ non-nested with $H$, as desired.
    \end{proof}

    \subsection{Canonical spanning trees}\label{sec:cycle_cutting_algorithm}

    We now present the Borel cycle-cutting algorithm that can be preformed on any countable median graph with finite hyperplanes. Applying this algorithm to the dual median graph of a finitely-separating family of cuts, which has finite hyperplanes by Proposition \ref{prp:dual_median_graph_of_cuts_has_finite_hyperplanes}, proves Theorem \ref{thm:component-wise_construction}.

    \begin{lemma}\label{lem:H_blocks_form_a_median_graph}
        For any subpocset $\mc{H}\subseteq\mc{H}_\mathrm{cvx}(X)$ of convex co-convex half-spaces in a median graph $(X,G)$, the principal orientations map $X\to\mc{U}^\circ(\mc{H})$ is surjective.
    \end{lemma}
    \begin{proof}
        Let $U\in\mc{U}^\circ(\mc{H})$, we need to find some $x\in X$ with $U=\widehat{x}$. Since $U\subseteq\mc{H}$ is clopen, there is a finite set $A\subseteq X$ $-$ which we may assume to be convex by Lemma \ref{lem:convex_of_finite_is_finite} $-$ such that for all $H\in\mc{H}$, we have $H\in U$ iff there is $H'\in U$ with $H\cap A=H'\cap A$. Note that $H\cap A\neq\em$ for every $H\in U$, since otherwise $\em\in U$. Furthermore, $H\cap H'\neq\em$ for every $H,H'\in U$, since otherwise we have $H\subseteq\lnot H'$, and so $\lnot H'\in U$.

        By Lemma \ref{lem:helly}, the intersection $(H\cap A)\cap(H'\cap A)=H\cap H'\cap A$ is non-empty, and applying it again furnishes some $x\in\bigcap_{H\in U}H\cap A$ in $X$. Thus $U\subseteq\widehat{x}$, so $U=\widehat{x}$ since both are orientations.
    \end{proof}

    This induces a \textit{$G$-adjacency} graph $X/\mc{H}\iso\mc{M}(\mc{H})$; explicitly, two $\mc{H}$-blocks $[x]_\mc{H},[y]_\mc{H}$ are $G$-adjacent if $(\widehat{x},\widehat{y})\in\mc{M}(\mc{H})$. Note that $\mc{M}(\mc{H})$ may be constructed as in Proposition \ref{prp:construction_of_dual_median_graph} since $\mc{H}\subseteq\mc{H}_\mathrm{cvx}(X)$ is finitely-separating by Lemma \ref{lem:half_space_separating_convex}. In particular, if $\mc{H}$ is nested, then $X/\mc{H}$ is a tree by Corollary \ref{cor:nested_implies_tree}.

    \begin{proposition}\label{prp:canonical_spanning_trees}
        If $(X,G)$ is a countable median graph with finite hyperplanes, then fixing any colouring of $\mc{H}_\mathrm{cvx}^\ast(X)$ into nested sub-pocsets yields a canonical spanning tree thereof.
    \end{proposition}
    \begin{proof}
        Such a colouring exists, since, by Corollary \ref{cor:non-nested_iff_embedding_of_hamming}, if two half-spaces $H,K\in\mc{H}^\ast_\mathrm{cvx}(X)$ are non-nested, then $\del_\mathsf{v}H\cap\del_\mathsf{v}K\neq\em$. Thus, the intersection graph of the boundaries admits a countable colouring, which descends into a colouring $\mc{H}^\ast_\mathrm{cvx}(X)=\bigsqcup_{n\in\N}\mc{H}^\ast_n$ such that each $H,\lnot H$ receive the same colour and that each $\mc{H}_n\coloneqq\mc{H}_n^\ast\cup\l\{\em,X\r\}$ is a \textit{nested} subpocset. For each $n\in\N$, let $\mc{K}_n\coloneqq\bigcup_{m\geq n}\mc{H}_m$.

        We shall inductively construct an increasing chain of subforests $T_n\subseteq G$ such that the components of $T_n$ are exactly the $\mc{K}_n$-blocks. Then, the increasing union $T\coloneqq\bigcup_nT_n$ is a spanning tree, since each $(x,y)\in G$ lies in a $\mc{K}_n$-block for sufficiently large $n$ (namely, the $n$ such that $\cone_x(y)\in\mc{H}^\ast_{n-1}$, since $\cone_x(y)$ and its complement are the only half-spaces separating $x$ and $y$ by Proposition \ref{prp:half-spaces_are_cones}).

        Since each pair of distinct points is separated by a half-space, the $\mc{K}_0$-blocks are singletons, so put $T_0\coloneqq\em$. Suppose that a forest $T_n$ is constructed as required. Note that each $\mc{K}_{n+1}$-block $Y\in X/\mc{K}_{n+1}$ is not separated by any half-spaces in $\mc{H}_m$ for $m>n$, but is separated by $\mc{H}_n$ into the $\mc{K}_n$-blocks contained in $Y$, which are precisely the $\mc{H}_n$-blocks in $Y/\mc{H}_n$. Pick an edge from the \textit{finite} hyperplane $\del_\mathsf{ie}H$ for each $H\in\mc{H}_n$, which connects a unique pair of $G$-adjacent blocks in $Y/\mc{H}_n$. Since each $Y/\mc{H}_n$ is a tree by Corollary \ref{cor:nested_implies_tree}, and each pair of $G$-adjacent blocks in $Y/\mc{H}_n$ is connected by a single picked edge, the graph $T_{n+1}$ obtained from $T_n$ by adding all such edges is a forest whose components are exactly the $\mc{K}_{n+1}$-blocks.
    \end{proof}

    \section{Borel Treeings of Graphings with Dense Cuts}\label{sec:borel_treeings_of_graphings_with_dense_cuts}

    We finally prove Theorem \ref{thm:treeing_quasi-trees}, stating that if a CBER $(X,E)$ admits a locally-finite graphing $G$ such that each component is a quasi-tree, then $E$ is treeable. The first step is to identify, for each component $G|C$, a family $\mc{H}(C)$ of finitely-separating cuts that is dense towards ends of $G|C$; since each $G|C$ is a quasi-tree, the cuts $\mc{H}(C)\coloneqq\mc{H}_{\diam(\del)\leq R_C}(C)\cap\mc{H}_\mathrm{conn}(C)$ for some $R_C<\infty$ from Section \ref{sec:graphs_with_dense_families_of_cuts} will do. Applying Theorem \ref{thm:component-wise_construction} then gives us, for each component $G|C$, a median graph $\mc{M}(\mc{H}(C))$ on $\mc{U}^\circ(\mc{H}(C))$ with finite hyperplanes.

    The issue lies in making the family $\mc{U}^\circ(\mc{H})\coloneqq\bigsqcup_C\mc{U}^\circ(\mc{H}(C))$ of \textit{all} clopen orientations on $\mc{H}\coloneqq\bigsqcup_C\mc{H}(C)$ into a standard Borel space. \textit{If $\mc{U}^\circ(\mc{H})$ is standard Borel}, the above partition induces a CBER $\mc{E}$ admitting a median graphing $\mc{M}(\mc{H})\coloneqq\bigsqcup_C\mc{M}(\mc{H}(C))$ with finite hyperplanes, from which one can implement the proof of Proposition \ref{prp:canonical_spanning_trees} in a Borel manner (using \cite{KM04}*{Lemma 7.3} for a countable colouring of the intersection graph of finite hyperplanes therein) to obtain a treeing of $\mc{E}$. Finally, $E$ is Borel bireducible with $\mc{E}$ via the principal orientations map $X\ni x\mapsto\widehat{x}\in\mc{U}^\circ(\mc{H})$, so $E$ is also treeable by \cite{JKL02}*{Proposition 3.3 (ii)}.

    We will remedy this issue using the fact that the cuts $\mc{H}(C)$ are dense towards ends of $G|C$. In particular, we have the following crucial lemma, which, by Proposition \ref{prp:construction_of_dual_median_graph}, shows that $\mc{M}(\mc{H}(C))$ is locally-finite.

    \begin{lemma}\label{lem:dual_median_graph_of_dense_cuts_locally_finite}
        Let $\mc{H}$ be a finitely-separating pocset of cuts on a connected locally-finite graph $(X,G)$. If $\mc{H}$ is dense towards ends, then each clopen orientation $U\in\mc{U}^\circ(\mc{H})$ contains finitely-many minimal cuts $H\in\mc{H}$.
    \end{lemma}
    \begin{proof}
        Fix a vertex $U\in\mc{U}^\circ(\mc{H})$ and let $\mc{K}\subseteq U$ be the minimal elements in $U$. Since $U\subseteq\mc{H}$ is clopen, there is a finite set $A\subseteq X$ such that for all $H\in\mc{H}$, we have $H\in U$ iff there is $H'\in U$ with $H\cap A=H'\cap A$. Note that $H\cap A\neq\em$ for every $H\in U$, for otherwise $\em\in U$; in particular, we have $\lnot H\in U$ for every $H\subseteq\lnot A$.

        Each end $p\in\epsilon(X)$ lies in $\lnot A$, so density of $\mc{H}$ furnishes $H_p\in\mc{H}$ with $p\in H_p\subseteq\lnot A$, and thus $\lnot H_p\in U$. Since $U$ is clopen, we have $K_p\subseteq\lnot H_p$ for some $K_p\in\mc{K}$. Thus $\l\{\lnot K_p\r\}$ covers $\epsilon(X)$, which by compactness contains a finite subcover $\l\{\lnot K_i\r\}_{i<n}$. We show that there are at-most finitely-many more minimal $K\in U$.

        Let $K\in\mc{K}\comp\l\{K_i\r\}_{i<n}$ be any other minimal element in $U$. By Lemma \ref{lem:connected_cuts_non_nested_finitely_others}, each $K_i$ is non-nested with finitely-many other half-spaces, so we may assume without loss of generality that $K$ is nested with every $K_i$. But $K\not\subseteq K_i\not\subseteq K$ and $K\cap K_i\neq\em$ for all $i<n$, so $\lnot K\subseteq\bigcap_{i<n}K_i\in\mc{H}_{\del<\infty}(X)$; the latter contains no ends in $\epsilon(X)$, so it is finite by Lemma \ref{lem:infinite_iff_contains_end}, and hence $\mc{K}$ is finite too.
    \end{proof}

    We now describe the encoding of $\mc{U}^\circ(\mc{H})$ into a standard Borel space. Since cuts have finite edge boundary, we may first represent the space $\mc{K}$ of \textit{all} non-trivial cuts of $G$ as a Borel subset of $[G]^{<\infty}$. The subcollection $\mc{H}\subseteq\mc{K}$ consisting of those cuts with component-wise bounded boundary diameter is also Borel, since each $R_C<\infty$ can be witnessed as the minimal number making $\mc{H}_{\diam(\del)\leq R_C}(C)$ dense towards ends of $G|C$, and the latter is a Borel condition as characterized in Proposition \ref{prp:dense_cuts_iff_proper_walling}. Finally, $\mc{U}^\circ(\mc{H})$ is a Borel subset of $[\mc{H}]^{<\infty}$, since we may encode each clopen orientation $U\in\mc{U}^\circ(\mc{H}(C))$ by its set of minimal elements in $\mc{H}(C)$, which is finite by Lemma \ref{lem:dual_median_graph_of_dense_cuts_locally_finite}. This makes $\mc{U}^\circ(\mc{H})$ a standard Borel space, and finishes the proof of Theorem \ref{thm:treeing_quasi-trees}.

    The above discussion actually proves the following generalization of Theorem \ref{thm:treeing_quasi-trees}, which is no longer about quasi-trees; rather, we only require that the (locally-finite) graphing admits a Borel family of `tree-like' cuts.

    \begin{theorem}\label{thm:treeing_graphings_with_dense_cuts}
        If a CBER $(X,E)$ admits a locally-finite graphing $G$ such that each component $G|C$ admits a family $\mc{H}(C)$ of finitely-separating cuts that is dense towards ends of $G|C$, and if $\mc{H}\coloneqq\bigsqcup_C\mc{H}(C)$ is a Borel subset of the standard Borel space of all cuts of $G$, then $E$ is treeable.
    \end{theorem}

\end{document}